\documentclass[12pt]{amsart}
\usepackage{graphicx}
\newtheorem{thm}{Theorem}[section]
\newtheorem{lem}[thm]{Lemma}

\newtheorem{prop}[thm]{Proposition}

\theoremstyle{definition}

\theoremstyle{remark}

\begin{document}

\title[Linear Type-$p$ Most-Perfect Squares]{Linear Type-$p$ Most-Perfect Squares}
\author{John Lorch}
\address{Department of Mathematical Sciences\\ Ball State University\\Muncie, IN  47306-0490}
\email{jlorch@bsu.edu}
\subjclass[2010]{05B30, 15B33}
\date{December 27, 2017}
\begin{abstract}
We describe a generalization of most-perfect magic squares, called type-$p$ most-perfect squares, and in prime-power orders we give a linear construction of these squares  reminiscent of de la Loub\`{e}re's classical magic square construction method. Type-$p$ most-perfect squares can be used to construct other interesting squares (e.g., generalized Franklin squares) and our linear construction may have implications for counting type-$p$ most-perfect  squares.
\end{abstract}
\maketitle

\def\a{{\bf a}}
\def\MF{M^{2\times 2}(\F )}
\def\rtimes{{\times\!\! |}}
\def\rk{{\rm rank}}
\def\A{{\mathcal A}}
\def\reg{{\mathcal R}}
\def\P{{\mathcal P}}
\def\End{{\rm End}}
\def\F{{\mathbb F}}
\def\calF{{\mathcal F}}
\def\GL{{\rm GL}}
\def\L{{\mathcal L}}
\def\R{{\mathbb R}}
\def\N{{\mathbb N}}
\def\C{{\mathbb C}}
\def\Z{{\mathbb Z}}
\def\B{{\mathbb B}}
\def\G{{\mathcal G}}
\def\g{\mathfrak g}
\def\OOA{{\rm OOA}}
\def\OA{{\rm OA}}
\def\S{{\mathcal S}}
\def\SA{{\rm SA}}
\def\SD{{\rm SD}}
\def\TD{{\rm TD}}
\def\Q{{\mathcal Q}}
\def\pt#1{{\langle  #1 \rangle}}
\def\ds{\displaystyle}
\def\v{{\bf v}}
\def\w{{\bf w}}

\makeatletter
\def\Ddots{\mathinner{\mkern1mu\raise\p@
\vbox{\kern7\p@\hbox{.}}\mkern2mu
\raise4\p@\hbox{.}\mkern2mu\raise7\p@\hbox{.}\mkern1mu}}
\makeatother

\section{Introduction}
Let $n$ be a natural number divisible by $p$. A natural pandiagonal magic square $R$ of order $n$ is said to be a {\bf most-perfect square of type-$p$} if the following two properties hold:
  \begin{itemize}
  \item[(i)] ({\bf Complementary property}) Starting from any location in $R$, consider the symbol in that location together with the $p-1$ other symbols lying in the same broken main-diagonal $n/p$ units apart from one another. The sum of these symbols is $\displaystyle \frac{p(n^2-1)}{2}$.
  \item[(ii)] ({\bf $p\times p$ property}) The symbols in any $p\times p$ subsquare formed from consecutive rows and columns (allowing wraparound) sum to $\displaystyle \frac{p^2(n^2-1)}{2}$.
  \end{itemize}

 Examples of type-$2$ and type-$3$ most-perfect squares are given in Figure \ref{f:mps8}.

\begin{figure}[h]
{\tiny
$$
\begin{array}{|cccc|cccc|}
\hline
 0 & 31 & 48 & 47 & 56 & 39 & 8 & 23 \\
 59 & 36 & 11 & 20 & 3 & 28 & 51 & 44 \\
 6 & 25 & 54 & 41 & 62 & 33 & 14 & 17 \\
 61 & 34 & 13 & 18 & 5 & 26 & 53 & 42 \\
 \hline
 7 & 24 & 55 & 40 & 63 & 32 & 15 & 16 \\
 60 & 35 & 12 & 19 & 4 & 27 & 52 & 43 \\
 1 & 30 & 49 & 46 & 57 & 38 & 9 & 22 \\
 58 & 37 & 10 & 21 & 2 & 29 & 50 & 45 \\
 \hline
\end{array}
\qquad
\begin{array}{|ccc|ccc|ccc|}
\hline
 0 & 16 & 23 & 63 & 79 & 59 & 45 & 34 & 41 \\
 64 & 80 & 57 & 46 & 35 & 39 & 1 & 17 & 21 \\
 47 & 33 & 40 & 2 & 15 & 22 & 65 & 78 & 58 \\
 \hline
 7 & 14 & 18 & 70 & 77 & 54 & 52 & 32 & 36 \\
 71 & 75 & 55 & 53 & 30 & 37 & 8 & 12 & 19 \\
 51 & 31 & 38 & 6 & 13 & 20 & 69 & 76 & 56 \\
 \hline
 5 & 9 & 25 & 68 & 72 & 61 & 50 & 27 & 43 \\
 66 & 73 & 62 & 48 & 28 & 44 & 3 & 10 & 26 \\
 49 & 29 & 42 & 4 & 11 & 24 & 67 & 74 & 60 \\
 \hline
\end{array}
$$}
    \caption{Left: A type-$2$ (classical) most-perfect square of order-8. Right: A type-$3$ most-perfect square
    of order $9$. The gridlines serve as an aid in locating complementary entries.}
    \label{f:mps8}
\end{figure}

In this article we use a linear method to construct certain type-$p$ most-perfect squares of order $p^r$, where $p$ is any prime.
In 1688 French diplomat Simon de la Loub\`{e}re returned from Thailand and wrote an account of his travels \cite{sL93}. In that account he relates a ``knight's move" magic square construction method he learned from his hosts; see \cite{wB87} for a description. We view De la Loub\'{e}re's method as a linear construction over a finite field.  Magic squares constructed by this method, or some suitable variation, are called {\bf linear magic squares}. This technique has shown promise in constructing new magic squares and rectangles (e.g., see \cite{jL14} and \cite{jL12}), and now we apply it to type-$p$ most-perfect squares.

Type-$p$ most-perfect squares specialize to classical most-perfect squares when $p=2$, in which case $n$ must be doubly even \cite{cP19}. The tasks of counting and constructing classical most-perfect squares were first approached by McClintock \cite{eM97} and culminate in the work of Ollerenshaw and Bree \cite{kO98}, which gives a count of the classical most-perfect squares for any doubly even order $n$, along with a construction method for all such squares. Also, classical most-perfect squares are useful in constructing Franklin magic squares (see \cite{pL06} and \cite{rN16}, and \cite{pP01} for historical background).

Aside being interesting in its own right, there are two primary motives for our linear construction of type-$p$ most-perfect squares. First, a generalized notion of Franklin magic squares is presented in \cite{jL18}. Such squares have orders that are triply divisible by some prime $p$, whereas the classical Franklin squares (more precisely, the ones of which we are aware) have orders triply divisible by $2$. Type-$p$ most-perfect squares can be used to produce examples of these generalized Franklin squares. Second, one may observe that all of the ten McClintock order-8 classical most-perfect squares are linear magic squares. These ten squares generate all most-perfect squares of order $8$  (\cite{eM97} and \cite{kO98}). This scant evidence suggests that it may be possible to cast the construction and counting of type-$p$ most-perfect squares of order $p^r$ in terms of linear squares.

\section{A Characterization of Certain Type-$p$ Most-Perfect Squares}
In this section we present a characterization of type-$p$ most-perfect squares with order divisible by $p^2$. Establishing the validity of our linear construction of such squares (Section \ref{s:construction}) is made easier through the use of this characterization.

\begin{lem}\label{l:diagsum}
Let $m,n\in \N$ and consider a nonnegative integer array $A$ of size  $(mp+1)\times (np+1)$ with
$$
A= \begin{array}{c|ccc|c}
a & \ & v & \ & b \\
\hline
\ & \ & \ & \ & \ \\
u & \ & D &\ & w\\
\ & \ & \ & \ & \ \\
\hline
c & \ & z & \ & d
\end{array}.
   $$
 Here $a,b,c,d\in \Z$, $u,w$ are lists of length $mp-1$, $v,z$ are lists of length $np-1$, and $D$ is an $(mp-1) \times (np-1)$ array. If $A$ possesses the $p\times p$ property then $a+ d=c+ b$.
\end{lem}

\begin{proof}
By the $p\times p$ property
$$
a+ u+ v+ D =b+ v+ w + D=c+ u+ z+ D =d+ z+ w+ D,
  $$
where the additions indicate the total sums of symbols in each type of list. It follows that
$$
(a+ u+ v+ D) + (d+ z+ w+ D)=(b+ v+ w + D)+ (c+ u+ z+ D ),
  $$
and cancellation gives the result.
\end{proof}

\begin{prop}\label{p:mpspcondition}
Let $n$ be a multiple of $p^2$. Any natural square of order $n$ that possesses both the complementary property and the $p\times p$ property is a most-perfect square of type $p$.
\end{prop}

\begin{proof}
Let $R$ be a square satisfying the hypotheses of the theorem. It suffices to show that $R$ is a pandiagonal magic square. The complementary property implies that broken main diagonals (i.e., translates--not cosets--of the main diagonal) achieve the magic sum. It remains to show that the same is true for rows, columns, and broken off-diagonals.

Let $\rho_ j$ ($0\leq j\leq n-1$) denote the integer sum of entries in the $j$-th row of $R$. By the complementary property we have
$$
\rho _0 + \rho _{n/p} + \rho _{2n/p}+ \cdots + \rho_{(p-1)n/p}= n\cdot \frac{p(n^2-1)}{2},
   $$
where in the righthand side the right term is the complementary sum, and $n$ is the number of entries per row. Meanwhile, because $n/p$ is a multiple of $p$, the $p\times p$ property implies that
$$
\rho _0 = \rho _{n/p} =\rho _{2n/p}=\cdots =\rho_{(p-1)n/p}.
   $$
Conclude that
$$
\rho _0 =\frac{1}{p}[\rho _0 + \rho _{n/p} + \rho _{2n/p}+ \cdots + \rho_{(p-1)n/p}]=\frac{n(n^2-1)}{2}.
   $$
Therefore the top row of $R$ has the magic sum; the same argument may be applied to any row or column of $R$.

It remains to  show that $R$ is pandiagonal. We do this by showing that $R$ must have the following off-diagonal complementary property:
Starting from any location in $R$, consider the symbol in that location together with the $p-1$ other symbols lying in the same broken off-diagonal $n/p$ units apart from one another. The sum of these symbols is $\displaystyle \frac{p(n^2-1)}{2}$.

 Our strategy is to  show  that each of these broken off-diagonal sums corresponds to a broken main-diagonal sum. This is evident when $p=2$, so we assume that $p$ is odd. Subdivide $R$ into $(n/p) \times (n/p)$ subsquares--there are $p^2$ such squares. Let $a_{i,j}$ represent the entry in the lower-left corner of the $(i,j)$-subsquare, with $0\leq  i,j\leq p-1$ counting from left to right and top  to bottom. Observe that
$$
S= a_{p-1,0}+ a_{p-2,1}+ \cdots + a_{0,p-1}
   $$
is an off-diagonal sum as described above in the off-diagonal complementary property. By Lemma \ref{l:diagsum} (which requires that $n/p$ be a multiple of $p$)  we have
\begin{align*}
S&=[a_{p-1,0}+ a_{1,p-2}]+ [a_{p-2,1}+ a_{2,p-3}]+ \cdots + [a_{(p+1)/2,(p-3)/2}+ a_{(p-1)/2,(p-1)/2}]+ a_{0,p-1} \\
&=[a_{1,0}+ a_{p-1,p-2}]+ [a_{2,1}+ a_{p-2,p-3}]+ \cdots + [a_{(p-1)/2,(p-3)/2}+ a_{(p+1)/2,(p-1)/2}]+ a_{0,p-1} .
  \end{align*}
  Observe that the latter sum in the equation above is a {\em main} broken diagonal complementary sum. Therefore, by the complementary property for main broken diagonals we have $S=\frac{p(n^2-1)}{2}$, as desired. A similar argument can be applied to any complementary broken off-diagonal sum.

The position of the terms in the complementary sums described above is illustrated in the following array:
{\tiny
$$
\begin{array}{|cc|cc|cc|cc||cc|}
\hline
\ & \ & \ & \ & \ & \ & \ & \ & \ & \ \\
\ & \ & \ & \ & \ & \ & \ & \ & a_{0,p-1} & \ \\
\hline \hline
\ & \ & \ & \ & \ & \ & \ & \ & \ & \ \\
a_{1,0} & \ & \ & \ & \ & \ & a_{1,p-2} & \ & \ & \ \\
\hline
\ & \ & \ & \ & \ & \ & \ & \ & \ & \ \\
\ & \ & a_{2,1} & \ & a_{2,p-3} & \ & \ & \ & \ & \ \\
\hline
\ & \ & \ & \ & \ & \ & \ & \ & \ & \ \\
\ & \ & a_{p-2,1} & \ & a_{p-2,p-3} & \ & \ & \ & \ & \ \\
\hline
\ & \ & \ & \ & \ & \ & \ & \ & \ & \ \\
a_{p-1,0} & \ & \ & \ & \ & \ & a_{p-1,p-2} & \ & \ & \ \\
\hline
\end{array}
  $$}
\end{proof}

\section{Magic Squares via Linearity}\label{ss:linearity}
In this section we show how linear transformations can be used to construct arrays of numbers, including magic squares. The essence of these ideas dates back to Simon de la Loub\`{e}re's classical method for constructing magic squares.

Locations in a $p^r \times p^r$ array can be described by elements of the vector space $\Z _p ^{2r}$. Rows are enumerated from the top, beginning with $0$ and ending with
$p^r -1$; columns are enumerated in the same way from left to right. By expressing each row number in base $p$ we can identify row locations with $\Z _p ^r$. Similarly we can identify column locations with $\Z _p ^r$, and therefore any grid location (row,column) may be identified with an element of $\Z _p ^{2r}$. By way of illustration, the symbol ``26" in the left portion of Figure \ref{f:mps8} lies in location $011101\in \Z _2 ^{2\cdot 3}$, where the first three entries indicate the row location and the last three entries indicate the column location.

Symbols in the set $S=\{0,1,\dots ,p^{2r}-1\}\subset \Z$ may be placed in a $p^r \times p^r$ array. These symbols  can also be described by elements of $\Z _p ^{2r}$. Each symbol $\Lambda $ has a unique base-$p$ expansion
$$
\Lambda =\lambda _{p^{2r-1}}p^{2r-1}+\lambda _{p^{2r-2}}p^{2r-2}+\cdots +\lambda _{p}\cdot p^1+\lambda _{1}\cdot p^0,
    $$
where $\lambda _{p ^j}\in\{0,1,\dots ,p-1\}$ for each $j\in \{0,1,\dots ,p^{2r-1}\}$. Therefore we can make the identification
$$
\Lambda\in\Z \longleftrightarrow (\lambda _{p^{2r-1}},\dots ,\lambda _{p^0})\in \Z _p^{2r}.
   $$

For example, if $p=2$ and $r=3$ as in the left portion of Figure \ref{f:mps8}, then the symbol  $\Lambda =26$ corresponds to $011010\in \Z_2^{2\cdot  3}$.

 A linear assignment of symbols to locations is as follows. Let $M$ be a $2r\times 2r$ matrix with entries in $\Z _p$. Define a linear mapping $T_M:\Z _p ^{2r}\rightarrow \Z _p ^{2r}$ by $T_M(\Lambda )=M\Lambda $. The mapping $T_M$ uniquely determines a $p^r\times p^r$ array with entries in $\{0,1,\dots ,p^{r+s}-1\}$ by declaring $T_M(\Lambda )$ to be the array location housing the number with base-$p$ representation $\Lambda $ (as described in the previous paragraphs). When $p=2$, $r=3$, and
\begin{equation}\label{e:thingone}
M=\left[
\begin{array}{cccccc}
 1 & 1 & 0 & 1 & 1 & 1 \\
 0 & 0 & 0 & 0 & 1 & 1 \\
 0 & 0 & 0 & 1 & 1 & 0 \\
 1 & 1 & 1 & 1 & 1 & 0 \\
 0 & 1 & 1 & 0 & 0 & 0 \\
 1 & 1 & 0 & 0 & 0 & 0 \\
\end{array}
\right],
   \end{equation}
the mapping $T_M$ determines the array in the left portion of Figure \ref{f:mps8}. To see that the number $26$ is sent the to the correct location, recall that $\Lambda =011010$ and so $M\Lambda =011101$, which is indeed the location housing $26$ as described above. Similarly, the matrix
\begin{equation}\label{e:thingtwo}
M=\left[
\begin{array}{cccc}
 2 & 2 & 2 & 0 \\
 0 & 0 & 1 & 1 \\
 2 & 0 & 2 & 2 \\
 1 & 1 & 0 & 0 \\
\end{array}
\right]
   \end{equation}
determines the order-$9$ array in the right portion of Figure \ref{f:mps8}.

Any magic square that possesses a linear representation as described above will be referred to as a {\bf linear magic square}.

\section{A Linear Construction of Most-Perfect Squares}\label{s:construction}
In this section we describe a linear construction of type-$p$ most-perfect squares of order $p^r$, where $r\geq 2$. In case $r=1$, a type-$p$ most-perfect square of order $p$ is simply a pandiagonal magic square of
prime order; de la Loub\`{e}re's method can be applied in that case as described in \cite{wB87}. We will be performing arithmetic over $\Z$ and over $\Z_p$; we occasionally let $\oplus$ denote addition over $\Z$ to distinguish it from addition over $\Z _p$.

For $r\geq 2$ let $\alpha_j=p^{2r-j}$ for $1\leq j\leq 2r$.  When considered as vectors in $\Z_p^{2r}$ as described in Section \ref{ss:linearity}, the set $\{\alpha _1,\dots ,\alpha _{2r}\}$ is the standard basis for $\Z_p^{2r}$.\footnote{Other bases for the space of symbols, such as a reordering the the $\alpha _j$'s, would work just as well and give different linear most-perfect magic squares.}  Let $L_r$ denote the $r\times r$ symmetric matrix with $1$'s on and below the off diagonal and $0$'s elsewhere. Therefore, if $L_r=(\ell _{ij})$ then
 $$
\ell _{ij}=\begin{cases}
1 & i+j> r \\
0 & i+j \leq r
\end{cases}
\quad \text{ and }\quad
L_r=\left[\begin{array}{cccc}
0 & 0 & \cdots & 1 \\
0 &      \vdots & \Ddots & 1\\
\vdots &1 &\cdots &1 \\
1 & 1 &\cdots &1
\end{array}\right].
   $$
Let $L$ denote the $2r\times 2r$ symmetric matrix matrix with block  form
$$
L=\left[\begin{array}{cc} 0 & L_r \\ L_r & 0\end{array}\right],
    $$
and let $\tilde L$ be the $2r\times 2r$ matrix with entries in $\Z_p$ whose columns $\tilde \ell _j$ satisfy
$$
\tilde \ell _j+(e_1+e_{r+1}) =\ell _j \quad 1\leq j\leq 2r,
   $$
where $e_1,\dots ,e_{2r}$ are the elementary vectors in $\Z_p^{2r}$. The matrix $\tilde L$ has the form
$$
\tilde L=
{\tiny
\left[\begin{array}{rrrrrrr|rrrrrrr}
-1 & -1 & -1 & -1 & \cdots & -1 & -1   & -1 & \cdots  & -1 & -1 & -1 & -1 & 0 \\
0& 0  & 0 & 0 & \cdots & 0 & 0   & 0 & \cdots  & 0 & 0 & 0 & 1 & 1 \\
0& 0  & 0 & 0 & \cdots & 0 & 0   & 0 & \cdots  & 0 & 0 & 1 & 1 & 1 \\
0& 0  & 0 & 0 & \cdots & 0 & 0   & 0 & \cdots  & 0 & 1 & 1 & 1 & 1 \\
0& 0  & 0 & 0 & \cdots & 0 & 0   & 0 & \cdots  & 1 & 1 & 1 & 1 & 1 \\
\vdots & \vdots  & \vdots & \vdots & \cdots & \vdots  & \vdots   & \vdots & \Ddots  & \vdots & \vdots & \vdots & \vdots & \vdots \\
0& 0  & 0 & 0 & \cdots & 0 & 0   & 1 & \cdots  & 1 & 1 & 1 & 1 & 1 \\
\hline
-1 & \cdots  & -1 & -1 & -1 & -1 & 0 &    -1 & -1 & -1 & -1 & \cdots & -1 & -1 \\
0 & \cdots  & 0 & 0 & 0 & 1 & 1 &    0& 0  & 0 & 0 & \cdots & 0 & 0 \\
0 & \cdots  & 0 & 0 & 1 & 1 & 1 &    0& 0  & 0 & 0 & \cdots & 0 & 0 \\
0 & \cdots  & 0 & 1 & 1 & 1 & 1 &    0& 0  & 0 & 0 & \cdots & 0 & 0 \\
0 & \cdots  & 1 & 1 & 1 & 1 & 1 &    0& 0  & 0 & 0 & \cdots & 0 & 0 \\
\vdots & \Ddots  & \vdots & \vdots & \vdots & \vdots & \vdots &  \vdots & \vdots  & \vdots & \vdots & \cdots & \vdots  & \vdots\\
1 & \cdots  & 1 & 1 & 1 & 1 & 1 &    0& 0  & 0 & 0 & \cdots & 0 & 0
\end{array}\right].
}
   $$
 Following Section \ref{ss:linearity}, we regard $\tilde L$ as a candidate for a matrix producing a linear most-perfect square of order $p^r$, with column $\tilde \ell _j$ being the location of symbol $\alpha _j$ for $1 \leq j\leq 2r$. The matrix $\tilde L$ is a good candidate for the following reasons: Observe that the columns of $L$ consist of the possible location vectors one can add to an existing location vector in order to  move one unit right or one unit down from that existing location. It is desirable to have control over these vectors to establish the $p\times p$ property in the resulting magic square. Meanwhile, translates of the location vector $e_1+e_{r+1}$ play an important role in the complementary property. The matrix $\tilde L$ gives some control over both of these features.

 Unfortunately, for reasons that will be made apparent as we proceed, $\tilde L$ won't quite work because aside from $r=2$ there is no symbol $\delta \in \Z_p^{2r}$ such that $\delta$ is nonzero in all of its components and
$\tilde L . \delta = e_1+e_{r+1}$. To fix this problem we modify $\tilde L$ to obtain a $2r\times 2r$ matrix $M$ over $\Z_p$ with columns $m_1,\dots ,m_{2r}$ satisfying
\begin{itemize}
\item $m_j=\tilde \ell _j$ for $ 1\leq j <r$ and $ r<j < 2r$.
\item $m_r=\tilde\ell _r +\ds \sum_{j=2}^{r-1} (-1)^{j+1} \tilde \ell_{r-j}$
\item $m_{2r}=\tilde \ell_{2r}+\ds\sum_{j=2}^{r-1} (-1)^{j+1} \tilde \ell_{2r-j}$
\end{itemize}
 If $r=2$ then $M=\tilde L$. If $r>2$ then $M$ agrees with $\tilde L$ in all but the $r$-th and $2r$-th columns, and those columns of $M$ are obtained from the respective columns of $\tilde L$ by applying elementary column operations on $\tilde L$. Further, if we put
\begin{equation}\label{e:delta}
\delta = \sum_{j=1}^{r}(-1)^{r+j} [e_{j}+e_{r+j}]
\end{equation}
then $\delta$ is non-zero in all of its components and $M.\delta = e_1+e_{r+1}$.

 The matrix $M$ has form shown below, where the undetermined elements depend on the parity of $r$ (left option when $r$ is odd, right option when
 $r$ is even):
$$
{\tiny
M=\left[\begin{array}{rrrrrrc|rrrrrrc}
-1 & -1 & -1 & -1 & \cdots & -1 & 0\ {\rm or}\ -1   & -1 & \cdots  & -1 & -1 & -1 & -1 & 1\ {\rm or}\ 0 \\
0& 0  & 0 & 0 & \cdots & 0 & 0   & 0 & \cdots  & 0 & 0 & 0 & 1 & 1 \\
0& 0  & 0 & 0 & \cdots & 0 & 0   & 0 & \cdots  & 0 & 0 & 1 & 1 & 0 \\
0& 0  & 0 & 0 & \cdots & 0 & 0   & 0 & \cdots  & 0 & 1 & 1 & 1 & 1 \\
0& 0  & 0 & 0 & \cdots & 0 & 0   & 0 & \cdots  & 1 & 1 & 1 & 1 & 0 \\
\vdots & \vdots  & \vdots & \vdots & \cdots & \vdots  & \vdots   & \vdots & \Ddots  & \vdots & \vdots & \vdots & \vdots & \vdots \\
0& 0  & 0 & 0 & \cdots & 0 & 0   & 1 & \cdots  & 1 & 1 & 1 & 1 & 0\ {\rm or}\ 1 \\
\hline
-1 & \cdots  & -1 & -1 & -1 & -1 & 1\ {\rm or}\ 0 &    -1 & -1 & -1 & -1 & \cdots & -1 & 0\ {\rm or}\ -1 \\
0 & \cdots  & 0 & 0 & 0 & 1 & 1 &    0& 0  & 0 & 0 & \cdots & 0 & 0 \\
0 & \cdots  & 0 & 0 & 1 & 1 & 0 &    0& 0  & 0 & 0 & \cdots & 0 & 0 \\
0 & \cdots  & 0 & 1 & 1 & 1 & 1 &    0& 0  & 0 & 0 & \cdots & 0 & 0 \\
0 & \cdots  & 1 & 1 & 1 & 1 & 0 &    0& 0  & 0 & 0 & \cdots & 0 & 0 \\
\vdots & \Ddots  & \vdots & \vdots & \vdots & \vdots & \vdots &  \vdots & \vdots  & \vdots & \vdots & \cdots & \vdots  & \vdots\\
1 & \cdots  & 1 & 1 & 1 & 1 & 0\ {\rm or}\ 1 &    0& 0  & 0 & 0 & \cdots & 0 & 0
\end{array}\right]
}
   $$

As described in Section \ref{ss:linearity}, let $R$ denote the square of order $p^r$ ($r\geq 2$) obtained by viewing $M$ as the matrix of the linear transformation carrying symbols $\{0,1,\dots , (p^r)^2-1\}$, each expressed as $2r$-tuples with respect to the basis $\{\alpha_1,\dots ,\alpha_{2r}\}$ and viewed as a member of $\Z _p^{2r}$, to locations determined by members of $\Z_p^{2r}$. Specific examples of $M$ are given in (\ref{e:thingone}) and (\ref{e:thingtwo}) in the cases $p=2$, $r=3$ and $p=3$, $r=2$, respectively. These $M$ give rise to the most-perfect squares $R$ shown in Figure \ref{f:mps8} (left and right, respectively).

We will show that $R$ is a most-perfect magic square by checking that $R$ satisfies the hypotheses of Proposition \ref{p:mpspcondition}. This will require several lemmas.

\begin{lem}\label{l:natural}
$R$ is a natural square.
\end{lem}

\begin{proof}
We show $M$ is nonsingular. Since $M$ is obtained from $\tilde L$ by elementary column operations, it suffices to show that
$\tilde L$ is nonsingular. If we replace the $(r+1)$-st row of $\tilde L$ by the sum of that row and the negative of the first row, we obtain a matrix $\hat L$ of the form
$$
\hat L=
{\tiny
\left[\begin{array}{rrrrrrr|rrrrrrr}
-1 & -1 & -1 & -1 & \cdots & -1 & -1   & -1 & \cdots  & -1 & -1 & -1 & -1 & 0 \\
0& 0  & 0 & 0 & \cdots & 0 & 0   & 0 & \cdots  & 0 & 0 & 0 & 1 & 1 \\
0& 0  & 0 & 0 & \cdots & 0 & 0   & 0 & \cdots  & 0 & 0 & 1 & 1 & 1 \\
0& 0  & 0 & 0 & \cdots & 0 & 0   & 0 & \cdots  & 0 & 1 & 1 & 1 & 1 \\
0& 0  & 0 & 0 & \cdots & 0 & 0   & 0 & \cdots  & 1 & 1 & 1 & 1 & 1 \\
\vdots & \vdots  & \vdots & \vdots & \cdots & \vdots  & \vdots   & \vdots & \Ddots  & \vdots & \vdots & \vdots & \vdots & \vdots \\
0& 0  & 0 & 0 & \cdots & 0 & 0   & 1 & \cdots  & 1 & 1 & 1 & 1 & 1 \\
\hline
0 & \cdots  & 0 & 0 & 0 & 0 & 1 &    0 & 0 & 0 & 0 & \cdots & 0 & 0 \\
0 & \cdots  & 0 & 0 & 0 & 1 & 1 &    0& 0  & 0 & 0 & \cdots & 0 & 0 \\
0 & \cdots  & 0 & 0 & 1 & 1 & 1 &    0& 0  & 0 & 0 & \cdots & 0 & 0 \\
0 & \cdots  & 0 & 1 & 1 & 1 & 1 &    0& 0  & 0 & 0 & \cdots & 0 & 0 \\
0 & \cdots  & 1 & 1 & 1 & 1 & 1 &    0& 0  & 0 & 0 & \cdots & 0 & 0 \\
\vdots & \Ddots  & \vdots & \vdots & \vdots & \vdots & \vdots &  \vdots & \vdots  & \vdots & \vdots & \cdots & \vdots  & \vdots\\
1 & \cdots  & 1 & 1 & 1 & 1 & 1 &    0& 0  & 0 & 0 & \cdots & 0 & 0
\end{array}\right].
}
   $$
By expanding along the top row of $\hat L$, we see that
$$\det (\hat L)=\pm \sum_{j=1}^{2r-1} (-1)^j=\pm 1 \ne 0.$$
We conclude that $M$ is nonsingular.
\end{proof}

\begin{lem}\label{l:pcomplementary}
$R$ possesses the complementary property.
\end{lem}

\begin{proof}
Let $v$ be location in $R$. We need to show that the symbols in $R$ with locations
$\{v+j\cdot (e_1+e_{r+1})\mid j\in \Z_p\}$  add to $\ds \frac{p(p^{2r}-1)}{2}$. Therefore, if $\nu = (\nu_1, \nu_2, \dots, \nu_{2r})\in \Z_p^{2r}$ with $M.\nu=v$, and $M.\delta =e_1+e_{r+1}$, we need to show that
\begin{equation}\label{e:diagsump}
\bigoplus_{j\in \Z_p} (\nu+j\cdot \delta):=\nu \oplus (\nu+\delta) \oplus (\nu+2\delta) \oplus \cdots \oplus (\nu+(p-1)\delta) =\frac{p(p^{2r}-1)}{2}.
  \end{equation}
Let $1\leq k\leq 2r$ and $\delta _k$ the $k$-th component of $\delta$. Because $\delta _k$ is nonzero in $\Z_p$, we know that upon rearrangement
$$
\bigoplus_{j\in \Z_p} (\nu_k+j\cdot \delta_k)=0\oplus 1\oplus \cdots \oplus (p-1)=\frac{p(p-1)}{2}.
   $$
Therefore, if we apply the base-$p$ addition algorithm in (\ref{e:diagsump}), we obtain
\begin{align*}
\bigoplus_{j\in \Z_p} (\nu+j\cdot \delta)&=\frac{p(p-1)}{2}(\alpha _1\oplus \alpha _2\oplus \cdots \oplus \alpha _{2r})\\
&=\frac{p(p-1)}{2}(p^{2r-1}\oplus p^{2r-2}\oplus \cdots \oplus p^1 \oplus 1)\\
&=\frac{p(p^{2r}-1)}{2},
  \end{align*}
as desired.
\end{proof}

\begin{lem}\label{l:pbypprop}
$R$ possesses the $p\times p$ property.
\end{lem}

\begin{proof}
It suffices to verify that if $A$ is any $(p+1)\times (p+1)$-subsquare of $R$ formed from consecutive rows and columns (allowing wraparound), with
$$
A=\begin{array}{|cccc|c|}
\hline
a_{11} & a_{12} & \dots & a_{1p} & a_{1,p+1}\\
a_{21} & a_{22} & \dots & a_{2p} & a_{2,p+1}\\
\vdots & \vdots & \vdots& \vdots &\vdots \\
a_{p1} & a_{p2} &\dots  & a_{pp} & a_{p,p+1} \\
\hline
a_{p+1,1} & a_{p+1,2} & \dots & a_{p+1,p} & a_{p+1,p+1}\\
\hline
\end{array},
  $$
then $\ds \bigoplus_{j=1}^p a_{1j} = \bigoplus_{j=1}^p a_{p+1,j}$ and $\ds \bigoplus_{j=1}^p a_{j1} = \bigoplus_{j=1}^p a_{j,p+1}$. (This is enough to show that all $p\times p$ subsquares in $R$ formed from consecutive rows and columns possess the same integer sum. That common integer sum is the sum of all symbols in $R$ divided by the number of $p\times p$ subsquares needed to tile $R$. Since $R$ is natural, this computation yields $\ds \frac{p^2(p^{2r}-1)}{2}$ for  the common sum of symbols in $p\times p$ subsquares, as  desired.)

We show that $\ds \bigoplus_{j=1}^p a_{j1}=\bigoplus _{j=1}^p a_{j,p+1}$. The other sum has a similar verification. We find the location for $a_{j1}$ by moving $j-1$ steps down from $a_{11}$, so
$$
a_{j1}=a_{11}+\sum_{i=2}^j d_i \quad \text{where}\quad M.d_i\in \{\ell _{r+1},\ell _{r+2},\dots ,\ell _{2r}\}.
   $$
Observe that for $r+1\leq j \leq 2r$
$$
\ell _j =e_1 + e_{r+1} +\tilde \ell _j =
\begin{cases}
(e_1+e_{r+1}) + m_j & \text{ if } r+1\leq j < 2r ,\\
(e_1+ e_{r+1})+ m_{2r}-\ds\sum_{k=2}^{r-1}(-1)^{k+1}m_{2r-k} & \text{ if } j=2r.
\end{cases}
   $$
From the construction of $R$ via $M$, it follows that $d_i=\delta +\beta _i$ where
$$
\beta _i\in\{\alpha_{r+1},\dots ,  \alpha _{2r-1}, \alpha _{2r}-\sum_{k=2}^{r-1}(-1)^{k+1}\alpha _{2r-k}\}\subseteq {\rm span}\{\alpha _{r+1},\dots ,\alpha_{2r}\},
   $$
and that
\begin{equation}\label{e:A1}
a_{j1}=(j-1)\delta +a_{11}+ \sum_{i=2}^j\beta _i.
\end{equation}
Meanwhile, we obtain $a_{j,p+1}$ by moving $p$ steps to the right from $a_{j1}$. Therefore
$$
a_{j,p+1}= \left[(j-1)\delta +a_{11}+ \sum_{i=2}^j\beta _i\right] +\sum_{i=2}^{p+1} r_i \quad \text{where}\quad  M.r_i\in\{\ell_1, \dots ,\ell _r\}.
   $$
Likewise it follows that $r_i=\delta +\gamma _i$ where
$\gamma _i\in \{\alpha _1, \dots ,\alpha _{r-1}, \alpha _{r}-\sum_{k=2}^{r-1}(-1)^{k+1}\alpha _{r-k}\}$. Therefore
\begin{equation}\label{e:A2}
\begin{split}
a_{j,p+1}&=\left[(j-1)\delta +a_{11}+ \sum_{i=2}^j\beta _i\right] +\sum_{i=2}^{p+1} (\delta +\gamma _i)\\
&=\left[(j-1)\delta +a_{11}+ \sum_{i=2}^j\beta _i\right] +\sum_{i=2}^{p+1} \gamma _i\\
&= \left[(j-1)\delta +a_{11}+ \sum_{i=2}^j\beta _i\right] +\gamma ,
\end{split}
   \end{equation}
where $\gamma =\sum_{i=1}^p \gamma _i\in {\rm span}\{\alpha _1,\dots ,\alpha _r\}$ and we recall $p\cdot \delta =0$ in $\Z_p$. To finish this verification it suffices to show equality of integer sums in each component. That is, we seek to show that $\ds \bigoplus_{j=1}^p (a_{j1})_k=\bigoplus _{j=1}^p (a_{j,p+1})_k$ for $1\leq k\leq 2r$. If $r+1\leq k\leq 2r$ then, because $\gamma \in {\rm span}\{\alpha _1,\dots, \alpha _r\}$, we have
$$
(a_{j,p+1})_k= \left[(j-1)\delta +a_{11}+ \sum_{i=2}^j\beta _i\right]_k = (a_{j1})_k.
  $$
It follows that $\ds \bigoplus_{j=1}^p (a_{j1})_k=\bigoplus _{j=1}^p (a_{j,p+1})_k$ for $r+1\leq k\leq 2r$. Meanwhile, if $1\leq k\leq r$ then
$(a_{j1})_k=[(j-1)\delta +a_{11}]_k$ and $(a_{j,p+1})_k=[(j-1)\delta +a_{11}+\gamma]_k$. Because $\delta _k$ is nonzero, upon reordering we have
$$
\bigoplus_{j=1}^p (a_{j1})_k=1\oplus 2\oplus \cdots \oplus p-1=\bigoplus_{j=1}^p (a_{j,p+1})_k.
  $$
\end{proof}

\begin{thm} \label{t:RMPSp}
The square $R$, as constructed above via $M$, is a linear type-$p$ most-perfect square of order $p^r$.
\end{thm}

\begin{proof}
By Proposition \ref{p:mpspcondition} it suffices to show that $R$ is natural, complementary, and possesses the $p\times p$ property. These characteristics are verified in Lemmas \ref{l:natural}, \ref{l:pcomplementary}, and \ref{l:pbypprop}, respectively.
\end{proof}


\begin{thebibliography}{99}

\bibitem{wB87}
W.\ W.\ R.\ Ball and H.\ S.\ M.\ Coxeter, {\em Mathematical Recreations and Essays}, thirteenth edition, Dover Publications, New York, 1987.

\bibitem{jL14}
J.\ Lorch, {\em Linear magic rectangles}, Linear Multilinear Algebra {\bf 62} no.\ 4 (2014), 530-537.

\bibitem{jL12}
J.\ Lorch, {\em Magic squares and sudoku}, Amer.\ Math.\ Monthly, {\bf 119} (2012), 759-770.

\bibitem{jL18}
J.\ Lorch, {\em Pandiagonal type-$p$ Franklin squares}, preprint.

\bibitem{sL93}
S.\ de la Loub\`{e}re, {\em A New Historical Relation of the Kingdom of Siam}, Printed by F.L.\ for Tho. Horne, London, 1693.


\bibitem{eM97}
E.\ McClintock, {\em On the most perfect forms of magic squares, with methods for their production}, Amer.\ J.\ Math.\
 {\bf 19} (1897), 99-120.

\bibitem{rN16}
R.\ Nordgren, {\em On Franklin and complete magic square matrices}, Fibonacci Quart.\ {\bf 54} no.\ 4 (2016), 304-318.

\bibitem{kO98}
Ollerenshaw and Bree, {\em Most Perfect Pandiagonal Squares}, Institute of Mathematics and its Applications, 1998.

\bibitem{pP01}
P.\ Pasles, {\em The lost squares of Dr.\ Franklin: Ben Franklin's missing squares and the secret of the magic circle}, Amer.\ Math.\ Monthly {\bf 108} (2001), 489-511.

\bibitem{cP19}
C.\ Planck, {\em Pandiagonal magic squares of order 6 and 10 with minimal numbers}, The Monist {\bf 29} (1919), 307-316.

\bibitem{pL06}
D.\ Schindel, M.\ Rempel, and P.\ Loly, {\em Enumerating the bent diagonal squares of Dr Benjamin Franklin FRS}, Proc.\ R.\ Soc.\ A {\bf 462} (2006), 2271-2279.
\end{thebibliography}
\end{document}